\documentclass{amsart}

\usepackage{amsmath}
\usepackage{amssymb}
\usepackage{amsthm}
\usepackage{mathrsfs}
\usepackage{mathtools}
\usepackage{tikz}
\usepackage{cleveref}
\usepackage{tikz-cd}
\usepackage{mathtools}

\usepackage{placeins}

\usepackage{float}

\usepackage{comment}

\theoremstyle{definition}
\newtheorem{thm}{Theorem}[section]
\crefname{thm}{Theorem}{Theorems}
\newtheorem{cor}[thm]{Corollary}
\newtheorem{prop}[thm]{Proposition}
\crefname{prop}{Proposition}{Propositions}
\newtheorem{lem}[thm]{Lemma}
\crefname{lem}{Lemma}{Lemmas}
\newtheorem{clm}[thm]{Claim}

\newtheorem{quest}[thm]{Question}
\newtheorem{defn}[thm]{Definition}

\crefname{defn}{Definition}{Definitions}

\newtheorem{exmp}[thm]{Example}

\newtheorem{rmk}[thm]{Remark}

\newtheorem*{ack*}{Acknowledgements}

\newcommand*{\on}[1]{\operatorname{#1}}

\title{Multicolour chain avoidance in the boolean lattice}
\author{Hunter Spink and Marius Tiba}

\begin{document}

\begin{abstract}
    Given a collection of colored chain posets, we estimate the number of colored subsets of the boolean lattice which avoid all chains in the collection.
\end{abstract}

\maketitle

\section{Introduction}
The method of hypergraph containers, recently introduced by Balogh, Morris and Samotij \cite{Balogh} and independently by Saxton and Thomason \cite{Saxton}, is an essential tool for counting independent sets in hypergraphs. Many natural problems can be phrased in this way, with the most direct applications toward determining the structure of graphs on $n$ vertices which avoid a collection of subgraphs. We refer the reader to the survey \cite{Survey}.

Recently, Balogh, Treglown and Wagner \cite{BaloghChain} showed that graph containers could be used in the boolean lattice $(\mathcal{P}([n]),\subseteq)$ to give an alternate proof of Kleitman's result \cite{Kleitman} counting the number of antichains in $\mathcal{P}([n])$. Using the hypergraph container algorithm, independently Collares and Morris \cite{MorrisChain} and Balogh, Mycroft, and Treglown \cite{MR3265918} were able to further count the number of antichains in a random subset of $\mathcal{P}([n])$, from which they were able to deduce the approximate size of the largest antichain therein.

Because every $k$-chain can be partitioned into $k-1$ antichains, the graph container lemma (i.e.~for $2$-uniform hypergraphs) suffice not only to count $k$-chain free sets, but also to create a small collection of small sized containers for $k$-chain free sets via a product construction. However, \cite{MorrisChain} directly constructs a set of hypergraph containers without exploiting this observation as an application of the recent advances in the hypergraph container lemma through balanced supersaturation results.

In this paper we answer analogous questions in a weighted coloured setting by building on \cite{MorrisChain}'s demonstration of the hypergraph container lemma through balanced supersaturation in $\mathcal{P}([n])$. Suppose we have colors $1,2,\ldots,m$, and a collection of colored chain posets $$\mathcal{G}:=\mathcal{G}_2\sqcup \mathcal{G}_3 \sqcup \ldots \sqcup \mathcal{G}_k$$ where $\mathcal{G}_i$ contains exclusively chains of order $i$. We say that a colored subset of the boolean lattice $(\mathcal{P}([n]),\subseteq)$ \emph{avoids all configurations from $\mathcal{G}$}, or is \emph{valid with respect to $\mathcal{G}$} if no element of $\mathcal{G}$ appears as a colored subchain. The present work is motivated by the following questions.

\begin{quest}
\label{Q1}
What is the cardinality of the collection $\Lambda(\mathcal{G},n)$ of validly colored subsets of $\mathcal{P}([n])$ with respect to $\mathcal{G}$?
\end{quest}
\begin{quest}
\label{Q2}
Let $p_1,\ldots,p_m\in [0,1]$ be such that $\sum_{i=1}^m p_i\le 1$. If we color each element of $\mathcal{P}([n])$ independently with color $c_i$ with probability $p_i$ and leave it uncolored with probability $1-\sum p_i$, then what is the expected number of validly colored subsets of $\mathcal{P}([n])$ with respect to $\mathcal{G}$?
\end{quest}

Multicoloured hypergraph container problems were only considered quite recently in the work of Falgas-Ravry, O'Connell, Str{\"o}mberg, and Uzzell \cite{Uzzell}. There it was shown that for a wide variety of colored configuration avoidance problems, if there is a validly coloured subset using all but a $o(1)$ proportion of the vertices, then the number of validly colored subsets can be estimated quite precisely (see \Cref{UzzellThm}).

The questions we consider are the first instances of colored hypergraph container problems that we are aware of which work in the presence of a \emph{sparse} extremal example. A separate interesting feature is that the hypergraph we work with is not uniform and no uniformity dominates, so when we iteratively apply the container algorithm we may have to use potentially different uniformities at each stage.

\begin{exmp}We now describe some instructive examples.
\begin{itemize}
\item Suppose that we have $1$ color, and suppose that we have only one forbidden chain $\mathcal{G}= \{(\underbrace{1 \prec 1 \prec \ldots \prec 1}_k)\}$. Then $\Lambda(\mathcal{G},n)$ is the collection of $k$-chain free sets, and by \cite{Kleitman,BaloghChain,MorrisChain} we have
$$|\Lambda(\mathcal{G},n)| = 2^{(k-1)\binom{n}{n/2}(1+o(1))}.$$
\item Suppose that we have $4$ colors, and let
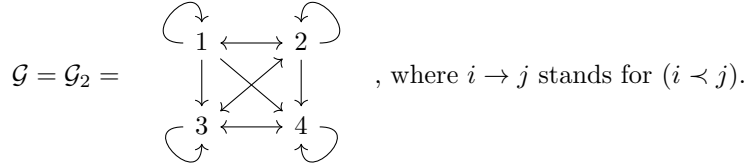
\begin{center}
$\mathcal{G}=\mathcal{G}_2=$
\begin{tikzcd}
1 \ar[d]\ar[out=180,in=90,loop,looseness=5]\ar[r,leftrightarrow]\ar[dr]&2\ar[d]\ar[out=0,in=90,loop,looseness=5]\\
3\ar[out=180,in=270,loop,looseness=5]\ar[r,leftrightarrow]\ar[ur,leftrightarrow]&4\ar[out=0,in=270,loop,looseness=5]
\end{tikzcd}, where $i \rightarrow j$ stands for $(i\prec j)$. 

Equivalently, the only allowed chains are
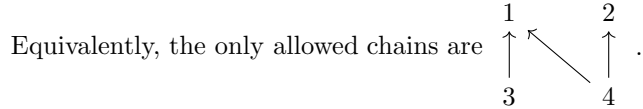
\begin{tikzcd}
1&2\\
3\ar[u]&4\ar[u]\ar[ul]
\end{tikzcd}
.
\end{center}
 
Notably in this example $\Lambda(\mathcal{G},n)$ contains two fundamentally different extremal families of configurations.
\begin{itemize}
    \item We can color all sets of size $\lfloor n/2 \rfloor$ with either $1$ or nothing, and all sets of size $\lfloor n/2 \rfloor-1$ with either $3,4$ or nothing.
    \item We can color all sets of size $\lfloor n/2 \rfloor$ with either $1,2$ or nothing, and all sets of size $\lfloor n/2 \rfloor-1$ with either $4$ or nothing.
\end{itemize}
In either case, the number of valid configurations formed is $$2^{\binom{n}{n/2}(1+o(1))}\cdot 3^{\binom{n}{n/2}(1+o(1))}=6^{\binom{n}{n/2}(1+o(1))}.$$
 As it turns out, by \Cref{mainthm} it follows that this is also an upper bound to $|\Lambda(\mathcal{G},n)|$.
\item Suppose we have $2$ colors, and
$\mathcal{G}=\mathcal{G}_2=\{(1 \prec 2)\}$. Then
$$|\Lambda(\mathcal{G},n)| \ge 2^{2^n(1+o(1))},$$ obtained by coloring each vertex either with $1$ or nothing. As we will shortly see, this is also an upper bound by \cite{Uzzell}.
\end{itemize}
\end{exmp}

In general, the cases of \Cref{Q1} and \Cref{Q2} when there is a \emph{dense} extremal example are solved by \cite{Uzzell}. This occurs exactly when $\mathcal{G}$ does not contain a monochromatic chain of every color. We are able to complete the analysis of these questions in the \emph{sparse} cases.

\subsection{Preliminary definitions and main results}
We recall from \cite{Uzzell} some basic definitions of multicolor hypergraphs in our context.

\begin{defn}
A \emph{template} is a function $$T:\mathcal{P}([n])\to \mathcal{P}(\{1,\ldots,m\}).$$ Say that a template $T$ is \emph{supported on $A\subset \mathcal{P}([n])$} if $T(x)=\emptyset$ whenever $x \not \in A$, and define $\on{Supp}(T)$ to be the smallest set on which $T$ is supported.
We say that a template is \emph{valid with respect to $\mathcal{G}$} if every coloring of $\on{Supp}(T)$ which assigns to each $x\in \on{Supp}(T)$ an element of $T(x)$ is valid with respect to $\mathcal{G}$. Say a coloured subset $A$ of $\mathcal{P}([n])$ is \emph{contained in $T$} if the color of every $x\in A$ lies inside $T(x)$.
Finally, denote by $$\omega(T)=\sum_{x \in \mathcal{P}([n])}\log(1+|T(x)|).$$
\end{defn}

The reason we consider valid templates is that they provide a lower bound on the number of valid configurations $$|\Lambda(\mathcal{G},n)| \ge \max_{T\text{ valid}}\prod_{x \in \mathcal{P}([n])}(1+|T(x)|)= e^{\max_{T\text{ valid}}\omega(T)}.$$  One would hope that this is the correct bound up to a $(1+o(1))$ factor in the exponent. Indeed, in the \emph{dense} case, we have the following theorem of \cite{Uzzell}.

\begin{thm}\cite{Uzzell}
\label{UzzellThm}
Define the \emph{maximal entropy} of $\mathcal{G}$ to be $$\pi(\mathcal{G})=\limsup_{n \to \infty}\frac{1}{2^n}\max_{T\text{ valid}} \omega(T).$$
We have
$$|\Lambda(\mathcal{G},n)|=e^{2^n(\pi(\mathcal{G})+o(1))}.$$
\end{thm}
As mentioned earlier, this theorem correctly estimates $|\Lambda(\mathcal{G},n)|$ up to a $1+o(1)$ factor in the exponent when $\pi(\mathcal{G})>0$. This happens precisely when there is a dense extremal example, i.e. $\mathcal{G}$ does not forbid a monochromatic chain of every color. However, when $\pi(\mathcal{G})=0$ the upper bound given by \Cref{UzzellThm} is trivial and does not estimate the correct exponent up to  a 1+o(1) factor. 

We solve \Cref{Q1} in the \emph{sparse} cases by estimating $|\Lambda(\mathcal{G},n)|$ up to a $1+o(1)$ factor in the exponent. This occurs when $\pi(\mathcal{G})=0$, i.e. $\mathcal{G}$ forbids a monochromatic chain of every color. Note that in this case, there exists $L=L(\mathcal{G})$ such that no validly colored set contains a chain of length $L$, and we may easily deduce a crude upper bound of $|\Lambda(\mathcal{G},n)|\le (m+1)^{(L-1)\binom{n}{n/2}(1+o(1))}$ by using the $L$-chain containers constructed from either of \cite{BaloghChain, MorrisChain}.

The following is one of our main theorems, solving the \emph{sparse} cases of \Cref{Q1}.

\begin{thm}
\label{mainthm}
Suppose that $\mathcal{G}$ forbids a monochromatic chain of every color, and define the constant (independent of $n$) $$\omega_{crit}=\max_T\{\omega(T)\mid T \text{ valid and supported on a chain poset}\}.$$ Then we have
$$|\Lambda(\mathcal{G},n)| = e^{\omega_{crit}\binom{n}{n/2}(1+o(1))}.$$
\end{thm}

Note that to compute the constant $\omega_{crit}$, we only need to evaluate $\omega(T)$ for $T$ ranging over the finite collection of valid templates supported on chain posets of length $L=L(\mathcal{G})$.

As we will see later in \Cref{largeprop}, we can in fact determine the exact maximum of $\omega(T)$ for valid templates $T$ on $\mathcal{P}([n])$, and the lower bound in \Cref{mainthm} will follow from considering such an extremal template.

The answer to \Cref{Q2} requires a weighted version of $\omega(T)$.

\begin{defn}
Given $\mathcal{G}$ that forbids a monochromatic chain of every colour and $\overline{\beta}=(\beta_1,\ldots,\beta_m)\in (\mathbb{R}_{>0})^m$, we denote by 
\begin{align*}
    |T(x)|_{\overline{\beta}}&=\sum_{i\in T(x)}\beta_i\\ 
\omega(\overline{\beta},T)&=\sum_{x \in \mathcal{P}([n])}\log(1+|T(x)|_{\overline{\beta}}) \\
\omega_{crit}(\overline{\beta})&=\max_T\{\omega(\overline{\beta},T)\mid T \text{ valid and supported on a chain poset}\}.\end{align*}
\end{defn}

The constant $\omega_{crit}(\overline{\beta})$ is also very easy to compute, and yields the critical exponent in the following theorem.

\begin{thm}
\label{expthm}
Suppose that $\mathcal{G}$ forbids a monochromatic chain of every color. With the probabilistic setup of \Cref{Q2}, denoting $\overline{p}=(p_1,\ldots,p_m)$ and $V$ the number of validly colored subsets of $\mathcal{P}([n])$ with respect to $\mathcal{G}$, we have
$$\mathbb{E}(V)=e^{\omega_{crit}(\overline{p})\binom{n}{n/2}(1+o(1))}.$$
\end{thm}
 As before, the methods from \cite{Uzzell} analogously answer the \emph{dense} cases where there is some color without a monochromatic forbidden chain. \Cref{expthm} thus completes the analysis of \Cref{Q2} in the \emph{sparse} case.

\section{Strategy and auxiliary results}
 In this section we outline the strategy of the proof and present some auxiliary results. We start by formulating a weighted version of \Cref{mainthm} and \Cref{expthm} which encompasses both of them. Throughout the rest of the paper we shall always assume that $\mathcal{G}$ forbids a monochromatic chain of each colour and refer to such a $\mathcal{G}$ as \emph{sparse}. In this section, we let $\mathcal{G}:=\mathcal{G}_2\sqcup \ldots \sqcup \mathcal{G}_k$ by a fixed collection of forbidden colored chains with colors $1,\ldots,m$, and we let
 $$\overline{\beta}=(\beta_1,\ldots,\beta_m)\in (\mathbb{R}_{>0})^m$$ be a fixed sequence of positive real weights.
 
 \begin{defn}
We define a measure $\mu$ on colored subsets of $\mathcal{P}([n])$ by assigning to a subset $S\subset\mathcal{P}([n])$ with a coloring $c:S\to \{1,\ldots,m\}$, the weight
 $$\mu(\overline{\beta},S)=\prod_{x \in S}\beta_{c(x)}$$
 and extending it additively. In particular, for a collection $\Lambda$ of colored subsets of $\mathcal{P}([n])$, we have
 $$\mu(\overline{\beta},\Lambda)=\sum_{S \in \Lambda}\mu(\overline{\beta},S).$$
 \end{defn}
Note that when $\beta_i=1$ for all $i$ we have that $\mu(\overline{\beta},S)=1$ for every colored subset $S$, and $\mu(\overline{\beta},\Lambda)=|\Lambda|$. We may now state the weighted reformulation of our main theorems.

 \begin{thm}
 \label{bigthm}
 Suppose that $\mathcal{G}$ is sparse, then we have
$$\mu(\overline{\beta},\Lambda(\mathcal{G},n)) = e^{\omega_{crit}(\overline{\beta})\binom{n}{n/2}(1+o(1))}.$$
 \end{thm}
 \begin{rmk}
\Cref{bigthm} specializes to \Cref{mainthm} when all $\beta_i=1$, and specializes to \Cref{expthm} when $\beta_i=p_i$.
 \end{rmk}
 
 \subsection{Proof of the lower bound of \Cref{bigthm}}
 
 We now prove two propositions which we shall use to prove the lower bound in \Cref{bigthm}. The first proposition relates the maximum weight of a template to the critical weight. The second proposition relates the weight of a template to the measure of the collection of colored subsets contained in the template. 
 
\begin{prop}
\label{measprop}
Given a template $T$, let $\Lambda_T$ be the collection of colored subsets contained in $T$. Then
$$\mu(\overline{\beta},\Lambda_T)=e^{\omega(\overline{\beta},T)}.$$
\end{prop}
\begin{proof}
Note that the right hand side is equal to
$$\prod_{x \in \on{Supp}(T)}(1+|T(x)|_{\overline{\beta}})=\prod_{x \in \on{Supp}(T)}(1+\sum_{i \in T(x)}\beta_i),$$
and expanding out the product yields the left hand side.
\end{proof}
 
\begin{prop}
\label{largeprop}
    The maximum value of $\omega(\overline{\beta},T)$ where $T$ is a valid template is attained for some $T$ with $\on{Supp}(T)$ a consecutive block of layers of $\mathcal{P}([n])$ which contains the middle layer, and with the property that $T(x)$ depends only on the size of $x$. In particular, we have that  $$\omega(\overline{\beta},T) = (1+O(\frac{1}{n})) \omega_{crit}(\overline{\beta})\binom{n}{n/2}.$$
\end{prop}

\begin{proof}
Our strategy will be to consider a valid template $T'$ which maximizes $\omega(\overline{\beta},T')$, and construct from it another valid template $T$ that satisfies the conclusions of \Cref{largeprop}. Choose a uniformly random maximal chain $C$ and consider the random variable
$$Z(C)=\sum_{x \in C} \binom{n}{|F|}\log(1+|T'(x)|_{\overline{\beta}}).$$
By linearity of expectation, it is easy to see that
$$\mathbb{E}Z(C)=\omega(\overline{\beta},T').$$ Therefore, there exists a chain $C$ such that $Z(C) \ge \omega(\overline{\beta},T')$. But then the template $T$ defined by $T(x)=T'(y)$ with $y \in C$ the unique element such that $|x|=|y|$ satisfies $$\omega(\overline{\beta},T)=Z(C)\ge \omega(\overline{\beta},T').$$ By construction $T$ is valid since $\mathcal{G}$ consists exclusively of chains, and by maximality of $T'$, $\omega(\overline{\beta},T)=\omega(\overline{\beta},T')$ is maximal. Clearly the maximality of such a $T$ further implies that $\on{Supp}(T)$ is a consecutive block of layers containing the middle layer.
\end{proof}

We are now ready to prove the lower bound in \Cref{bigthm}.

\begin{proof}[Proof of the lower bound in \Cref{bigthm}]
Let $T$ be the extremal template from \Cref{largeprop} and let $\Lambda_T$ be the collection of colored subsets contained in $T$. Then by \cref{measprop} and \cref{largeprop} we have that

$$\mu(\overline{\beta},\Lambda(\mathcal{G},n)) \ge \mu( \overline{\beta},\Lambda_T) = e^{\omega(\overline{\beta},T)}= e^{(1+O(\frac{1}{n})) \omega_{crit}(\overline{\beta})\binom{n}{n/2}}, $$

which gives the desired lower bound.
\end{proof}

\subsection{Proof of the upper bound of \Cref{bigthm} assuming a balanced supersaturation result}

Now we describe the outline of the proof of the upper bound in \Cref{bigthm}. The main thrust of the proof is identical to that of \cite{MorrisChain}. The key new ideas are to create a balanced supersaturation result that works for templates and to implement the container lemma in a way which handles simultaneously the various uniformities of $\mathcal{G}$.

Our goal will be to find a collection $\mathcal{C}$ of $e^{o(1)\binom{n}{n/2}}$ templates with each template $T \in \mathcal{C}$ having $w(\overline{\beta},T) \le (\omega_{crit}(\overline{\beta})+o(1))\binom{n}{n/2}$ such that every validly colored subset of $\mathcal{P}([n])$ is contained in some template $T \in \mathcal{C}$. Then by a union bound we can conclude \cref{bigthm}.

To accomplish this, we will use the following hypergraph container lemma. 

Given a hypergraph $\mathcal{H}$, we denote $v(\mathcal{H})$ for the vertices of $\mathcal{H}$, $e(\mathcal{H})$ for the edges of $\mathcal{H}$, and we recall for $A \subset v(\mathcal{H})$ the standard notations $d_{\mathcal{H}}(A)$ for the number of hyperedges of $\mathcal{H}$ which contain $A$, and $j$th codegree $\Delta_j(\mathcal{H})=\max_{|A|=j}d_{\mathcal{H}}(A)$.

\begin{lem} \cite{Balogh,Saxton}
\label{hyplem}
For every $K \in \mathbb{N}$ and $c>0$ there exists $\epsilon>0$ such that the following holds. Let $\tau \in (0,1)$ and suppose that $\mathcal{H}$ is a $K$-uniform hypergraph on $N$ vertices such that for $1 \le j \le K$ we have
\begin{align*}
    \Delta_j(\mathcal{H})\le& c\cdot \tau^{j-1}\frac{e(\mathcal{H})}{N}.
\end{align*}
Then there exists a family $\mathcal{C}$ of subsets of $v(\mathcal{H})$, and a function $f:\mathcal{P}(v(\mathcal{H}))\to \mathcal{C}$ such that
\begin{enumerate}
    \item For every $I \in \mathcal{I}(\mathcal{H})$, there exists $F \subset I$ with $|F| \le K\tau N$ and $I \subset F \cup f(F)$
    \item $|C| \le (1-\epsilon)N$ for every $C \in \mathcal{C}$.
\end{enumerate}
\end{lem}

In order to use the hypergraph container lemma we translate between validly colored subsets of $\mathcal{P}([n])$ and independent subsets of a certain hypergraph.  We consider the following ambient non-uniform hypergraph $\mathcal{A}$ defined by
$$v(\mathcal{A})=\mathcal{P}([n])\times \{1,\ldots,m\},$$
\begin{align*}
    e(\mathcal{A})=\bigcup_{\ell=1}^k\{\left((x_1,i_1),\ldots, (x_\ell,i_\ell)\right) \mid x_1 \subsetneq \ldots \subsetneq x_{\ell}\text{ in $\mathcal{P}([n])$ and}\\
(i_1 \prec \ldots \prec i_\ell) \in \mathcal{G}_{\ell}\}.
\end{align*}
By construction, a validly colored subset of $\mathcal{P}([n])$ can be viewed as an independent set in $\mathcal{A}$, though we remark that this is not a 1-1 correspondence. Also, there is a natural $1$-$1$ correspondence between templates $T$ and subsets of $v(\mathcal{A})$, where we assign to a template $T$ the subset of all $(x,c)$ with $c \in T(x)$.
By a slight abuse of notation we shall sometimes view $T$ as a subset of vertices in $v(\mathcal{A})$ and sometimes view $T$ as the induced sub-hypergraph
$$T=\mathcal{A}|_T\subset \mathcal{A}.$$
The notion of order of the sub-hypergraph associated to $T$ is related to the notion of weight of $T$ by
$$|T|=|v(\mathcal{A}|_T)|=\Theta( \omega(\overline{\beta},T)),$$
i.e. it is within a constant factor of $\omega(\overline{\beta},T)$.

 Our desired set of hypergraph containers will correspond to a family of templates which efficiently contains validly colored subsets of $\mathcal{P}([n])$. In what follows we shall usually notate $\mathcal{H}$ for a sub-hypergraph of $T$.

\begin{lem}
\label{balancelem}
Suppose that $\mathcal{G}_k$ contains all colored chains of length $k$. For every $\alpha>0$ there exists a $\delta>0$ such that the following holds. Let $n \in \mathbb{N}$ and suppose that $T$ is a template of $\mathcal{P}([n])$ supported on sets of size between $\frac{n}{3}$ and $\frac{2n}{3}$ such that $\omega(\overline{\beta},T) \ge (\omega_{crit}(\overline{\beta})+\alpha)\binom{n}{n/2}$. Then there exists $2\le l \le k$ and there exists an $l$-uniform sub-hypergraph $\mathcal{H}_\ell$ of $T$ such that
\begin{align*}e(\mathcal{H}_\ell) &\ge \delta^ln^{l-1}\binom{n}{n/2}\text{, and}\\
\Delta_j(\mathcal{H}_\ell) &\le (\delta n)^{l-j}
\end{align*}
for $1 \le j \le l$.
\end{lem}

\begin{cor}
\label{bigcor}
For every $\alpha \in (0,1)$ there exists an $\epsilon>0$ such that the following holds. Let $n \in \mathbb{N}$ and suppose that $T$ is a template of $\mathcal{P}([n])$ supported on sets of size between $\frac{n}{3}$ and $\frac{2n}{3}$ such that $\omega(\overline{\beta},T) \ge (\omega_{crit}(\overline{\beta})+\alpha)\binom{n}{n/2}$. Then there exists a family $\mathcal{C}$ of sub-templates of $T$ such that
\begin{enumerate}
    \item For every validly colored subset $I$ contained in $T$, there is a $T' \in \mathcal{C}$ such that $I$ is contained in $T'$.
    \item We have $|\mathcal{C}|\le e^{O(1)\frac{\log n}{n}|T|}$ for some constant $O(1)$ independent of $\alpha$.
    \item We have $|T'|\le (1-\epsilon)|T|$ for every $T' \in \mathcal{C}$.
\end{enumerate}
\end{cor}
\begin{proof}[Proof of \Cref{bigcor} assuming \Cref{balancelem}]
It suffices to prove this when $\mathcal{G}_k$ contains all colored chains of length $k$ (as we can always augment $\mathcal{G}$ with all colored chains of length $km$ without changing the valid configurations). Partition $T\subset v(\mathcal{A})$ into sets $T_0\cup T_1 \cup \ldots \cup T_r$ for some $r\ge 1$ such that \begin{align*}\omega(\overline{\beta},T_0)&<(\omega_{crit}(\overline{\beta})+\alpha)\binom{n}{n/2},\text{ and}\\ \omega(\overline{\beta},T_i)&=(\omega_{crit}(\overline{\beta})+\alpha)\binom{n}{n/2}+O(1).\end{align*} By \cref{balancelem}, there exists $\delta=\delta(\alpha)$ such that for each of $1\le i \le r$ there exists an $\ell_i$-uniform sub-hypergraph $\mathcal{H}^i=\mathcal{H}^i_{\ell_i}$ of the templates $T_i$ for some $2 \le \ell_i \le k$ with the property that
\begin{align*}
    e(\mathcal{H}^i)&\ge \delta^{\ell_i}n^{\ell_i-1}\binom{n}{n/2},\text{ and}\\
    \Delta_j(\mathcal{H}^i) &\le (\delta n)^{\ell_i-j}\text{, for all $1 \le j \le \ell_i$.}
\end{align*}
Let $2 \le l \le k$ be the most frequent uniformity. Construct the $\ell$-uniform sub-hypergraph $\mathcal{H}$ of $T$ with $v(\mathcal{H})=v(T)$ and $$e(\mathcal{H}) : = \bigsqcup_{l_{i}=l} e(\mathcal{H}^{i}).$$

By construction we have \begin{align*}
    |e(\mathcal{H})| &\ge \frac{r}{k} \delta^{l}n^{l-1}\binom{n}{n/2}\\
    |v(\mathcal{H})| &= |v(T)|= {{\sum_{i=0}^r|v(T_i)|=O(1)\sum_{i=0}^r\omega(\overline{\beta},T_i)}}=O(1)r\binom{n}{n/2}\\
    \Delta_j(\mathcal{H}) &\le (\delta n)^{\ell-j}\text{ for all $1 \le j \le \ell$.}
\end{align*}

Set $\tau= \frac{1}{n}$ and $c= O(1)k\delta^{-k}$, and apply \cref{hyplem} to the $l$-uniform hypergraph $\mathcal{H}$ to obtain the following. There exists $\epsilon$ depending only on $c,k$ and there exists a collection $\mathcal{C}$ of subtemplates $T'$ of $T$, and a function $f:\mathcal{P}(v(T))\to \mathcal{C}$ such that
\begin{enumerate}
    \item For every $I \in \mathcal{I}(T)$, there exists $F \subset I$ with $|F| \le \frac{k}{n}|T|$ and $I \subset F \cup f(F)$
    \item $|T'| \le (1-\epsilon)|T|$ for every $T' \in \mathcal{C}$.
\end{enumerate}

Set $\mathcal{C}':=\{F \cup f(F) \text{ : } F \in \mathcal{C}\}$ and note that

\begin{enumerate}
    \item For every $I \in \mathcal{I}(T)$, there exists $T' \in \mathcal{C}'$ such that $I \subset T'$
    \item $|\mathcal{C}'|\le \frac{k}{n}|T|\binom{|T|}{\frac{k}{n}|T|}=e^{O(1)\frac{\log(n)}{n}|T|}$.
     \item $|T'| \le (1-\epsilon+o(1))|T|$ for every $T' \in \mathcal{C}'$.
\end{enumerate}
\end{proof}
\begin{proof}[Proof of upper bound of \Cref{bigthm} assuming \Cref{bigcor}]
Let $\mathcal{P}([n])'$ be all vertices in $x\in \mathcal{P}([n])$ with $|x| \in (\frac{n}{3},\frac{2n}{3})$, and let $\Lambda'(\mathcal{G},n)$ be the validly colored subsets of $\mathcal{P}([n])'$. We have $$\mu(\overline{\beta},\Lambda(\mathcal{G},n))\le \left(\prod_{x \not \in \mathcal{P}([n])'}(1+\sum_{i=1}^m \beta_i)\right)\mu(\overline{\beta},\Lambda'(\mathcal{G},n))=e^{o(1)\binom{n}{n/2}}\mu(\overline{\beta},\Lambda'(\mathcal{G},n)).$$
Therefore it is enough to prove that for every $\alpha>0$ we have
$$\mu(\overline{\beta},\Lambda'(\mathcal{G},n))\le e^{(\omega_{crit}(\overline{\beta})+\alpha+o(1))\binom{n}{n/2}}.$$

Fix a threshold value $1>\alpha>0$. Starting with $\mathcal{A}|_{\mathcal{P}([n])'\times \{1,\ldots,m\}}$, we iteratively apply \Cref{bigcor} until we obtain a family $\mathcal{C}$ of subtemplates $T$ with $\omega(\overline{\beta},T) \le (\omega_{crit}(\overline{\beta})+\alpha)\binom{n}{n/2}$. This is encoded by a branching process where a template $T$ with $\omega(T)\ge (\omega_{crit}(\overline{\beta})+\alpha)\binom{n}{n/2}$ splits into subtemplates $T_1,T_2,\ldots,T_s$ such that each validly colored subset contained in $T$ is contained in some $T_i$. By \Cref{bigcor}, there exists an $\epsilon=\epsilon(\alpha)>0$ such that we have $s \le e^{O(1)\frac{\log(n)}{n}|T|}$, and $|T_i|\le (1-\epsilon)|T|$.

Because our initial set has size at most $m2^n$, and the size of the templates decreases by a factor of $(1-\epsilon)$ each iteration, each template at level $i$ in this branching process splits into at most $e^{O(1)\frac{\log(n)}{n}(1-\delta)^im2^n}$ other templates. Therefore, the final collection $\mathcal{C}$ of templates has cardinality bounded above by
$$|\mathcal{C}|\le \prod_{i=0}^{\infty}e^{O(1)\frac{\log(n)}{n}(1-\epsilon)^im2^n}=e^{O(1)\frac{log(n)}{n}\epsilon^{-1}m2^n}=e^{o(1)\binom{n}{n/2}}.$$
Note that each set in $\Lambda'(\mathcal{G},n)$ is contained in some $T\in \mathcal{C}$. Therefore, letting $\Lambda_T$ be the collection of colored subsets contained in $T$ we have by \Cref{measprop}
$$\mu(\overline{\beta},\Lambda'(\mathcal{G},n))\le \sum_{T \in \mathcal{C}}\mu(\overline{\beta},\Lambda_T)=\sum_{T \in \mathcal{C}}e^{\omega(\overline{\beta},T)}\le e^{o(1)\binom{n}{n/2}}e^{(\omega_{crit}(\overline{\beta})+\alpha)\binom{n}{n/2}}.$$
\end{proof}

\section{Balanced Supersaturation}
In this section, we prove \Cref{balancelem}. To do this, we prove a series of technical results adapted from \cite{MorrisChain} for our purposes. We fix a template $T$ for the remainder of this section, and recall that by hypothesis $\mathcal{G}_k$ contains all colored chains of length $k$.

\begin{defn}
Define the following random variables on a uniformly chosen random maximal chain $C$ in $\mathcal{P}([n])$. Let
$$X(C)=\sum_{x \in C}\log(1+|T(x)|_{\overline{\beta}}),$$ and $Y(C)$ be the total number of colored subchains of $C$ contained in $T$ which appear as a colored chain in $\mathcal{G}$.
\end{defn}

\begin{defn}
For $x \in \mathcal{P}([n])$ define the constant $X^x = \log(1+|T(x)|_{\overline{\beta}})$. Define the following random variables on a uniformly chosen random maximal length chain $C$ in $\mathcal{P}([n])$ whose top element is $x$. Let $Y^x(C)$ be the number of colored subchains of $C$ contained in $T$ whose top element is $x$ and appear as a colored chain in $\mathcal{G}$.
\end{defn}


\begin{lem}
\label{averagelem}
There are constants $C_1,C_2>0$ independent of $n$ such that the following is true. For any $\alpha \in (0,C_2)$, if $T$ is a template with ${{\omega(\overline{\beta},T)}}\ge (\omega_{crit}(\overline{\beta})+\alpha)\binom{n}{n/2}$, then there exists a vertex $x\in \on{Supp}(T)$ such that
$$\mathbb{E}Y^x \ge C_1\alpha.$$
\end{lem}

\begin{proof}
Take $C_3=\log(1+\sum_{i=1}^m \beta_i)$, $C_4=\log(1+\min(\beta_i))$ and take $C_1,C_2$ to be \begin{align*}
C_1&=C_3^{-1}C_4\frac{1}{2\omega_{crit}(\overline{\beta})}\\
    C_2&=\min\{C_3^{-1}C_1^{-1}\log(1+\min(\beta_i)),\omega_{crit}(\overline{\beta})\}.
\end{align*} First note that
$$X(C)-C_3Y(C) \le \omega_{crit}(\overline{\beta}).$$
Indeed, while $X(C)>\omega_{crit}(\overline{\beta})$ we can find a forbidden colored subchain of $C$ contained in $T$. Deleting one by one vertices of $C$ from forbidden subchains, $Y(C)$ decreases each time by at least $1$ and $X(C)$ decreases each time by at most $C_3$.

Now suppose for the sake of contradiction that $\mathbb{E}Y^x < C_1\alpha$ for all $x\in \on{Supp}(T)$. If for any $x\in \on{Supp}(T)$ we have $X^x \le C_3\mathbb{E}Y^x$, we obtain the contradiction (recalling $\alpha<C_2$)
$$\log(1+\min(\beta_i))\le X^x \le C_3\mathbb{E}Y^x < C_3C_1\alpha<\log(1+\min(\beta_i)).$$
Hence we have $X^x > C_3\mathbb{E}Y^x$ for all $x\in\on{Supp}(T)$. Writing $X$ and $Y$ as a sum of indicator functions and using linearity of expectation we have $\mathbb{E}(X)=\sum_{x \in \on{Supp}(T)}\frac{1}{\binom{n}{|x|}}X^x$, and $\mathbb{E}Y=\sum_{x \in \on{Supp}(T)}\frac{1}{\binom{n}{|x|}}\mathbb{E}Y^x$. Thus we obtain the contradiction

\begin{align*}
    \omega_{crit}(\overline{\beta})&\ge \mathbb{E}(X-C_3Y)=\sum_{x \in \on{Supp}(T)}\frac{1}{\binom{n}{|x|}}(X^x-C_3\mathbb{E}Y^x)\\&\ge \sum_{x \in \on{Supp}(T)}\frac{1}{\binom{n}{n/2}}(X^x-C_3\mathbb{E}Y^x)
    \ge\frac{\omega(T)}{\binom{n}{n/2}}-\frac{|\on{Supp}(T)|}{\binom{n}{n/2}}C_3C_1\alpha\\
    &\ge \frac{\omega(T)}{\binom{n}{n/2}}(1-C_4^{-1}C_3C_1\alpha)\ge (\omega_{crit}(\overline{\beta})+\alpha)(1-\frac{1}{2\omega_{crit}(\overline{\beta})}\alpha)\\&=\omega_{crit}(\overline{\beta})+\frac{1}{2}\alpha(1-\frac{\alpha}{\omega_{crit}(\overline{\beta})})\\
    &> \omega_{crit}(\overline{\beta}).
\end{align*}

\end{proof}

The following lemma, inspired by a corresponding lemma from \cite{MorrisChain} (adapted from an argument of \cite{Sudakov}),
gives us very good control over the number of colored chains below a given vertex. It is surprising that given our non-transitive family $\mathcal{G}$ of forbidden chains that we still retain such excellent control.

\begin{lem}
\label{boundlem}
There is a constant $Q\ge 0$ independent of $n$ such that the following is true. For any $x \in X$, $i \le k$, and $C$ a chain of maximal length whose top element is $x$, let $Z^x_{c_1 \succ \ldots \succ c_i}(C)$ be the number of colored subchains of $C$ of order $i$ {{contained in $T$}}, whose top element is $x$ and is colored $c_1 \succ \ldots \succ c_i$. Then
$$Z^x_{c_1\succ \ldots \succ c_i}(C) \le Q+Y^x(C).$$
\end{lem}
\begin{proof}
Recall that $\mathcal{G}_k$ contains all chains of length $k$. Therefore, for $i=k$ any $Q\ge 0$ works. If $i<k$, it is enough to ensure that $Q$ satisfies
$$Z^x_{c_1 \succ \ldots \succ c_i}(C) \le Q+\sum_{j=1}^m Z^x_{c_1 \succ \underbrace{\scriptstyle j \succ \ldots \succ j}_{k-1}}(C).$$
Indeed, if we can show this then by using the trivial bound $$  Q+\sum_{j=1}^m Z^x_{c_1 \succ \underbrace{\scriptstyle j \succ \ldots \succ j}_{k-1}}(C) \le Q+Y^x_k(C),$$ the conclusion follows immediately. 

If $c_1\not\in T(x)$ then the result is trivially true for any choice of $Q\ge 0$, so we assume that $c_1$ is a valid choice at $x$. Denoting $s=|C\cap\on{Supp}(T)|$, we have the trival bounds $$Z^x_{c_1 \succ \ldots \succ c_i}(C) \le (s-1)^{i-1} \text{ and }
\sum_{i=1}^m Z^x_{c_1 \succ \underbrace{\scriptstyle j \succ \ldots \succ j}_{k-1}}(C) \ge \binom{\frac{s-1}{m}}{k-1},$$
where the second bound follows by observing that the most frequent colour on $C$ appears at least $\frac{s-1}{m}$ times.

Thus it suffices to take $Q$ such that
$$(s-1)^{i-1}\le Q+\binom{\frac{s-1}{m}}{k-1}$$ for every $1\le i<k$, and every $s \ge 1$.
\end{proof}

\begin{proof}[Proof of \Cref{balancelem}]
We build an auxiliary sub-hypergraph $\mathcal{H}$ of $T$ one edge at a time, ensuring with each new edge that $\Delta_j(\mathcal{H}_\ell)\le (\delta n)^{\ell-j}$ holds for all $\ell$ and $1 \le j \le \ell$, until for some $\ell$ we have $e(\mathcal{H}_\ell)\ge \delta^{\ell-1}n^\ell\binom{n}{n/2}$, and then we output $\mathcal{H}_\ell$.\footnote{This is analogous to \cite{MorrisChain}, except since they worked in a single uniformity they could directly construct their final hypergraph without using an auxiliary hypergraph.}  In particular we assume that at the current stage $e(\mathcal{H}_\ell)<\delta^{\ell-1}n^\ell \binom{n}{n/2}$ for all $\ell$. Note that given a colored chain $B$ contained in $T$ that also appears in $\mathcal{G}$, if it cannot be added to $\mathcal{H}$, then there exists a colored subchain $B'$ satisfying $d_{\mathcal{H}_\ell}(B')=(\delta n)^{\ell-|B'|}$ for some $\ell$, so adding $B$ to $\mathcal{H}$ would violate the codegree condition. We will implicitly find such a colored chain $B$ which we can add to $\mathcal{H}$ by constructing it one vertex at a time from the top down, ensuring that no codegree condition among the subsets of $B$ is violated at each step. The following claim shows that there are very few ways of extending a ``good'' $B$ to a ``bad'' $B$ with the addition of a vertex.


\begin{clm}
\label{boundclm}
Given an $i$-chain $x_1 \supsetneq \ldots \supsetneq x_i$ with $x_j$ colored by $c_j \in T(x_j)$, then there are at most $O(1) \delta n$ choices for $(x_{i+1},c_{i+1})$ with $x_{i}\supsetneq x_{i+1}$ and $c_{i+1}\in T(x_{i+1})$ such that there exists some $2 \le \ell \le k$, and \emph{nonempty} $A \subset \{(x_1,c_1),\ldots,(x_i,c_i)\}$ with $|A|\le \ell-1$ such that $d_{\mathcal{H}_\ell}(A \cup \{(x_{i+1},c_{i+1})\})=(\delta n)^{\ell-(|A|+1)}$.
\end{clm}
\begin{proof}
Fix some $A \subset \{(x_1,c_1),\ldots,(x_i,c_i)\}$, $2 \le \ell \le k$,  and let $\mathcal{B}$ be the set of all $(x_{i+1},c_{i+1})$ with $x_{i+1}\subsetneq x_i$ and $c_{i+1} \in T(x_{i+1})$ and $d_{\mathcal{H}_\ell}(A \cup \{(x_{i+1},c_{i+1})\})=(\delta n)^{\ell-(|A|+1)}$. The disjoint union
$$\bigsqcup_{(x_{i+1},c_{i+1})\in \mathcal{B}}\text{hyperedges of $\mathcal{H}_\ell$ containing $A \cup \{(x_{i+1},c_{i+1})\}$}$$
has size $(\delta n)^{\ell-(|A|+1)}$, and each edge appears at most $k$ times. Ignoring repeats, this is a collection of hyperedges of $\mathcal{H}_\ell$ containing $A$, so has at most $\Delta_{|A|}(\mathcal{H}_\ell) \le (\delta n)^{\ell-|A|}$ distinct elements. Therefore, we conclude that $|\mathcal{B}| \le k\delta n$. Summing over all choices of $A,\ell$ we obtain the desired result.
\end{proof}

We remark that the claim (via the nonemptiness condition on $A$) does not take into account the possibility that the new colored vertex $v$ added to $B$ violates the $\Delta_1(\mathcal{H}_\ell)$ condition for the singleton $\{v\}$. However, we will explicitly deal with this possibility by disregarding such colored vertices.

We continue the proof of \Cref{balancelem} along the lines of \cite{MorrisChain}, which we include for completeness (and rephrase in terms of random variables for convenience). By a double counting argument, the number of colored vertices $v$ with $d_{\mathcal{H}_\ell}(\{v\})=(\delta n)^{\ell-1}$ is at most $\ell e(\mathcal{H}_\ell)/(\delta n)^{\ell-1}\le \ell \delta \binom{n}{n/2}$. Omitting these colored vertices, and assuming we take $\delta<\alpha/(2\ell\log(1+\sum_{i=1}^m \beta_i))$, we obtain a template $T'$ with $\omega(\overline{\beta},T')\ge (\omega_{crit}+\frac{\alpha}{2})\binom{n}{n/2}$. We will now show there is a colored chain contained in $T'$ which we can add to $\mathcal{H}$. For the remainder of the proof we take all random variables with respect to $T'$ rather than $T$.

By \Cref{averagelem} applied to $T'$, we can take $x\in \on{Supp}(T')$ minimal such that $\mathbb{E}Y^x \ge C_1\frac{\alpha}{2}$. Then for $x' \subsetneq x$, and $c_1,\ldots,c_j$ with $j \le k$ by \Cref{boundlem} and the minimality of $x$ we have
$$\mathbb{E}Z^{x'}_{c_1 \succ \ldots \succ c_j} \le Q+\mathbb{E}Y^{x'} \le Q+C_1\frac{\alpha}{2}$$
and
$$\mathbb{E}Z^x_{c_1 \succ \ldots \succ c_j} \le Q+\mathbb{E}Y^x.$$
Our goal is to find a colored chain with $x$ as its uncolored top element contained in $T'$ that we can add to $\mathcal{H}$ without violating any of the codegree conditions. To do this, we write $Y^x=Y^x_{bad}+Y^x_{good}$ where $Y^x_{bad}$ only counts colored chains contained in $T'$ we are not allowed to add to $\mathcal{H}$. It suffices to prove the upper bound
\begin{align*}\mathbb{E}Y^x_{bad}&\le \sum_\ell\sum_{c_1,\ldots,c_l}\sum_{1 \le i \le \ell} \mathbb{E}Z^x_{c_1 \succ \ldots \succ c_i}\cdot \left(O(1)\delta n  \frac{1}{n/3}\right)\cdot\max_{x'\subsetneq x}\mathbb{E}Z^{x'}_{c_{i+1}\succ \ldots \succ c_\ell}.
\end{align*}

Indeed, by the above the right hand side is bounded above by $(Q+\mathbb{E}Y^x)(O(1)\delta)(Q+C_1\frac{\alpha}{2})$, and by choosing $\delta$ sufficiently small in terms of $\alpha$ and the absolute constant $Q$ (independent of $\mathbb{E}Y^x$), we can guarantee this is strictly less than $\mathbb{E}Y^x$ (using the fact that $\mathbb{E}Y^x \ge C_1\frac{\alpha}{2}$). Therefore $Y^x_{good}$ is not identically zero and we can find a new hyperedge to add to $\mathcal{H}$.

To do this we first similarly split $Z^x_{c_1,\ldots,c_\ell}=Z^x_{c_1,\ldots,c_\ell,bad}+Z^x_{c_1,\ldots,c_\ell,good}$, and write $$Y^x_{bad}(C)=\sum_\ell\sum_{(c_1\succ\ldots\succ c_\ell)\in \mathcal{G}_\ell}Z^x_{c_1,\ldots,c_\ell,bad}(C).$$
Next, we upper bound
$$Z^x_{c_1,\ldots,c_\ell,bad}(C)\le \sum_i\sum_{x_i \subsetneq x}\sum _{x_{i+1}\subsetneq x_i}Z^x_{c_1,\ldots,c_\ell,bad,i,x_i,x_{i+1}}(C)$$ where $Z^x_{c_1,\ldots,c_\ell,bad,i,x_i,x_{i+1}}$ counts those colored subchains of $C$ contained in $T'$, whose top element is $x$, colored $c_1 \succ \ldots \succ c_\ell$, such that the $i$'th and $i+1$'st elements from the top are precisely at the locations $x_i,x_{i+1}$ respectively, and furthermore that $x_{i+1}$ along with some subset of the colored elements of the chain above it violate some codegree condition. 

To bound the expectation of the right hand side of this triple sum, we first note that
\begin{align*}\mathbb{E}Z^x_{c_1,\ldots,c_\ell,bad,i,x_i,x_{i+1}}&=(\mathbb{E}Z^{x_{i+1}}_{c_{i+1}\succ \ldots \succ c_\ell})\mathbb{E}Z^x_{c_1,\ldots,c_{i+1},bad,i,x_i,x_{i+1}}\\&\le (\max_{x' \subsetneq x}\mathbb{E}Z^{x'}_{c_{i+1}\succ \ldots \succ c_\ell})\cdot \mathbb{E}Z^x_{c_1,\ldots,c_{i+1},bad,i,x_i,x_{i+1}}\end{align*}
By \Cref{boundclm}, we now have
$$\sum_{x_{i+1}\subsetneq x_i}\mathbb{E}Z^x_{c_1,\ldots,c_{i+1},bad,i,x_i,x_{i+1}} \le O(1)\delta n \frac{1}{n/3}\mathbb{E}Z^x_{c_1,\ldots,c_i,x_i}$$
where $Z^x_{c_1,\ldots,c_i,x_i}(C)$ counts those colored subchains of $C$ contained in $T'$, whose top element is $x$, colored $c_1 \succ \ldots \succ c_i$, such that the bottom element is $x_i$. 

Finally, note that $$\sum_{x_i\subsetneq x} Z^x_{c_1,\ldots,c_i,x_i}(C)=Z^x_{c_1 \succ \ldots \succ c_i}(C).$$ 

Putting this all together now yields the desired inequality.

\end{proof}

\bibliography{refs}
\bibliographystyle{alpha}

\end{document}